\documentclass[table]{amsart}
\usepackage[margin=1.1in]{geometry}
\usepackage{amscd,amsmath,amsxtra,amsthm,amssymb,xr,mathrsfs,mathtools,enumerate,orcidlink}
\usepackage{tikz}
\usetikzlibrary{arrows.meta,positioning}

\usetikzlibrary{calc}
\usepackage{tikz-3dplot}
\usepackage{stmaryrd}
\usepackage{multirow}
\usepackage{xcolor}
\usepackage{commath}
\usepackage{comment}
\usepackage{svg}
\usepackage{graphics}
\usepackage{longtable} 
\usepackage{pdflscape} 
\usepackage{booktabs}
\usepackage{hyperref}
\definecolor{vegasgold}{rgb}{0.77, 0.7, 0.35}
\definecolor{darkgoldenrod}{rgb}{0.72, 0.53, 0.04}
\definecolor{gold(metallic)}{rgb}{0.83, 0.69, 0.22}
\hypersetup{
 colorlinks=true,
 linkcolor=darkgoldenrod,
 filecolor=brown,      
 urlcolor=gold(metallic),
 citecolor=darkgoldenrod,
 }
\newtheorem{lthm}{Theorem}

\usepackage[all,cmtip]{xy}

\DeclareFontFamily{U}{wncy}{}
\DeclareFontShape{U}{wncy}{m}{n}{<->wncyr10}{}
\DeclareSymbolFont{mcy}{U}{wncy}{m}{n}
\DeclareMathSymbol{\Sh}{\mathord}{mcy}{"58}
\usepackage[T2A,OT2,T1]{fontenc}

\newtheorem{theorem}{Theorem}[section]
\newtheorem{lemma}[theorem]{Lemma}

\newtheorem*{theorem*}{Theorem}
\newtheorem*{ass*}{Assumption}
\newtheorem{definition}[theorem]{Definition}

\newcommand{\Frac}{\mathrm{Frac}}

\newcommand{\Z}{\mathbb{Z}}
\newcommand{\Q}{\mathbb{Q}}

\newcommand{\cO}{\mathcal{O}}

\newcommand{\op}[1]{\operatorname{#1}}
\newcommand\mtx[4] { \left( {\begin{array}{cc}
 #1 & #2 \\
 #3 & #4 \\
 \end{array} } \right)}

 \DeclareMathSymbol{\sha}{\mathord}{mcy}{"58}
 \makeatletter
\newcommand{\mylabel}[2]{#2\def\@currentlabel{#2}\label{#1}}
\makeatother

\numberwithin{equation}{section}

\title[class groups of cubic orders with prescribed shape]{On the average $2$-torsion in class groups and narrow class groups of cubic orders with prescribed shape}
\author[A.~Ray]{Anwesh Ray\, \orcidlink{0000-0001-6946-1559}}
\address[Ray]{Chennai Mathematical Institute, H1, SIPCOT IT Park, Kelambakkam, Siruseri, Tamil Nadu 603103, India}
\email{anwesh@cmi.ac.in}

\begin{document}

\maketitle

\begin{abstract}
We study the distribution of $2$-torsion in class groups and narrow class groups
of cubic fields and cubic orders subject to prescribed shape conditions. The
\emph{shape} of a cubic order in a number field is a natural geometric invariant that takes values in the modular
surface $\mathbb{H}/\operatorname{GL}_2(\mathbb{Z})$. Fix a subset $W$ of the modular surface with positive hyperbolic measure whose boundary has measure $0$. Refining the methods of Bhargava and Varma, we prove that
among cubic fields with shape in $W$, the average size of the $2$-torsion subgroup
of the class group is $5/4$ for totally real fields and $3/2$ for complex fields,
while the average size of the $2$-torsion subgroup of the narrow class group for
totally real cubic fields is $2$. Similar results are obtained for cubic orders satisfying prescribed local conditions at all primes.
\end{abstract}

\section{Introduction}
\subsection{Motivation and historical context} The arithmetic statistics of number fields have advanced significantly in recent years, much of this work being motivated by the Cohen--Lenstra heuristics, which predict the distribution of class groups of number fields of fixed degree and signature. Some of the fundamental results in this direction concern the average sizes of torsion subgroups. Davenport and Heilbronn \cite{davenportheilbronn} proved that, as $K$ varies over real (resp. imaginary) quadratic fields, the average size of $\op{Cl}_3(K)$ is $4/3$ (resp. $2$). Likewise, Bhargava \cite{bhargavequartic} proved that, for totally real (resp. complex) cubic fields $K$, the average size of $\op{Cl}_2(K)$ is $5/4$ (resp. $3/2$). These results are driven in large part by powerful parametrization techniques originating in the work of Delone--Faddeev \cite{DeloneFaddev} and their far-reaching generalizations due to Bhargava \cite{bhargavahighercomplaws2}.
\par Alongside these developments, there has been increasing interest in finer
geometric invariants of number fields. One such invariant is the \emph{shape} of
a number field, defined using the Minkowski embedding of its ring of integers. Equipping $K\otimes_{\mathbb{Q}}\mathbb{R}$ with its positive-definite Minkowski inner product, the image of
$\mathcal{O}_K$ becomes a Euclidean lattice containing the vector corresponding to $1$.
Projecting orthogonally to this vector yields a rank-$(n-1)$ lattice, whose
equivalence class up to scaling and orthogonal transformations defines the
shape. The space of all such shapes may be identified with the double coset
space
\[
\mathcal{S}_{n-1}
=
\operatorname{GL}_{n-1}(\mathbb{Z})
\backslash
\operatorname{GL}_{n-1}(\mathbb{R})
/
\operatorname{GO}_{n-1}(\mathbb{R}),
\]
and carries a natural measure induced from Haar measure. Distribution questions for shapes were initiated by Terr \cite{Terr97}, who proved that the
shapes of cubic fields are equidistributed in $\mathcal{S}_2$. In this case,
$\mathcal{S}_2$ may be identified with the modular surface
$\mathbb{H}/\op{GL}_2(\mathbb{Z})$, where $\mathbb{H}$ is the complex upper half-plane. Equidistribution holds
with respect to the hyperbolic measure $y^{-2}dxdy$. More recently, Bhargava and Harron \cite{bhargavaharron}
established analogous equidistribution results for $S_n$-number fields of
degrees $n=4,5$.
\subsection{Main results}
In this work, we combine these perspectives by studying averages of $2$-torsion
in class groups and narrow class groups of cubic orders subject to prescribed
shape conditions. We show that imposing such geometric restrictions does not
alter the limiting averages. The expected size of the relevant $2$-torsion
groups of cubic fields with shape contained in a suitably nice subset $W\subseteq \mathcal{S}_2$ coincides with the averages obtained by Bhargava and Varma \cite{BhargavaVarma}. The following is a refinement of \cite[Theorem 1]{BhargavaVarma}.
\begin{lthm}\label{thm a}
    Let $W \subset \mathcal{S}_2$ be a measurable subset with positive measure whose boundary is of measure
zero. Fix any finite set of primes and impose arbitrary
local conditions at these primes. When cubic fields with shape lying in $W$ and
satisfying these local conditions are ordered by absolute discriminant, the
following statements hold:
\begin{itemize}
\item[(a)] Among totally real cubic fields, the average size of the
$2$-torsion subgroup of the class group is $5/4$.
\item[(b)] Among complex cubic fields, the average size of the
$2$-torsion subgroup of the class group is $3/2$.
\item[(c)] Among totally real cubic fields, the average size of the
$2$-torsion subgroup of the narrow class group is $2$.
\end{itemize}

\end{lthm}

We also obtain results for all cubic orders. Let $(\Sigma_p)$ be an \emph{acceptable collection} of local specifications as in Definition \ref{acceptable Sigma}. Given an order $\cO$ in a cubic number field $K$, let $\mathcal{I}(\cO)$ be the group of invertible fractional ideals of $\cO$ and $\mathcal{P}(\mathcal{O})\subseteq \mathcal{I}(\cO)$ be the subgroup of principal fractional ideals. Denote by $\mathcal{P}^+(\cO)
\subseteq \mathcal{P}(\mathcal{O})$ the subgroup of totally positive principal fractional ideals. Then the class group (resp. narrow class group) of $\cO$ is the quotient $\op{Cl}(\cO):=\mathcal{I}(\cO)/\mathcal{P}(\mathcal{O})$ (resp. $\op{Cl}^+(\cO):=\mathcal{I}(\cO)/\mathcal{P}^+(\mathcal{O})$). Given a natural number $n$, denote by $\op{Cl}_n(\cO)$ (resp. $\mathcal{I}_n(\cO)$) the $n$-torsion subgroup of $\op{Cl}(\cO)$ (resp. $\mathcal{I}(\cO)$). We have the following refinement of \cite[Theorem 2]{BhargavaVarma}.

\begin{lthm}\label{thm b}
    Let $(\Sigma_p)$ be an acceptable collection of local specifications, and let
$\Sigma$ denote the set of isomorphism classes of cubic orders
$\mathcal{O}$ such that $\mathcal{O}\otimes\mathbb{Z}_p \in \Sigma_p$ for all
primes $p$. Let $W \subset \mathcal{S}_2$ be a measurable subset with positive measure and whose boundary has
measure zero. When cubic orders $\mathcal{O}\in\Sigma$ whose shapes
lie in $W$ are ordered by absolute discriminant, the following hold:
\begin{itemize}
\item[(a)] For totally real cubic orders $\mathcal{O}$, the average value of
\[
\left|\mathrm{Cl}_{2}(\mathcal{O})\right|
-\tfrac{1}{4}\left|\mathcal{I}_{2}(\mathcal{O})\right|
\]
is $1$.
\item[(b)] For complex cubic orders $\mathcal{O}$, the average value of
\[
\left|\mathrm{Cl}_{2}(\mathcal{O})\right|
-\tfrac{1}{2}\left|\mathcal{I}_{2}(\mathcal{O})\right|
=
\left|\mathrm{Cl}_{2}^{+}(\mathcal{O})\right|
-\tfrac{1}{2}\left|\mathcal{I}_{2}(\mathcal{O})\right|
\]
is $1$.
\item[(c)] For totally real cubic orders $\mathcal{O}$, the average value of
\[
\left|\mathrm{Cl}_{2}^{+}(\mathcal{O})\right|
-\left|\mathcal{I}_{2}(\mathcal{O})\right|
\]
is $1$.
\end{itemize}

\end{lthm}
\noindent We note that Theorem \ref{thm b} is simply a restatement of Theorem \ref{main thm of article}. These results have many interesting consequences for cubic fields with prescribed shape in a set $W\subseteq \mathcal{S}_2$. For instance, the following is a refinement of \cite[Corollary 3]{BhargavaVarma}.
\begin{lthm}\label{thm c}
    Let $W \subset \mathcal{S}_2$ be a measurable subset with positive measure and whose boundary has
measure zero. Fix a finite set of primes and impose arbitrary local conditions at these primes. When cubic fields with shape lying in $W$ and satisfying these local conditions
are ordered by absolute discriminant, the following statements hold:
\begin{itemize}
\item[(1)] A positive proportion of totally real cubic fields have odd class
number; in fact, this proportion is at least $75\%$.
\item[(2)] A positive proportion of complex cubic fields have odd class number;
in fact, this proportion is at least $50\%$.
\end{itemize}
\end{lthm}

\subsection{Organization}
The structure of the paper is as follows. In Section~2 we recall the notion of the shape of a lattice and of a number field, with particular emphasis on the cubic case. We also review the parametrizations of cubic rings and $2$-torsion ideal classes that will be used throughout the paper. In Section~3 we carry out the necessary volume computations, adapting the framework in \cite{bhargavaharron}. These calculations differ from those in \emph{loc. cit.} where the shape distribution is studied for quartic orders, while in our applications, the parameterization for $2$-torsion ideal classes in cubic orders is considered. The emphasis lies in the $\op{GL}_2$-action on the homogeneous space $\mathbb{R}^2\otimes\op{Sym}^2(\mathbb{R}^3)$ as opposed to the $\op{GL}_3$-action. In Section~4 we incorporate local congruence conditions and establish equidistribution of shapes for acceptable families of cubic orders. Finally, we combine these results to prove our main theorems on the average size of $2$-torsion in class groups and narrow class groups under shape restrictions. Theorems \ref{thm a} and \ref{thm c} are proven at the end of this section.
\subsection{Outlook}
The results of this paper lead to some natural questions. One may
ask whether the Cohen--Lenstra distribution for class groups persists when
restricted to number fields of fixed degree $d$ and signature whose shapes lie in a measurable
region $W \subset \mathcal{S}_{d-1}$ of positive measure and boundary of measure
zero. The present work provides affirmative evidence for this principle in the
case of $2$-torsion for cubic fields.

\subsection*{Declarations}
\begin{description}
     \item[Conflict of interest] Not applicable; there is no conflict of interest to report.
    \item[Availability of data and materials] No data was generated or analyzed in obtaining the results in this article.
    \item[Funding] No funding was received during the preparation of this article.
    \item[Use of artificial intelligence]
The author used ChatGPT 6 Astra to assist in editing the original submission while preparing the first revision of the article. The AI tool was used to help identify and correct errors in the manuscript. It was not used to generate the mathematical ideas or the main results of the article.
\end{description}

\section{Preliminary notions}

\subsection{Shapes of lattices}

The notion of the \emph{shape} of a full rank lattice in Euclidean space is an invariant which captures its intrinsic geometry up to the natural symmetries. To formalize this notion, fix a real vector space $V$ of dimension $r$ with ordered basis $e_1, \dots, e_r$. Assume that $V$ is equipped with the standard inner product defined by $\langle e_i, e_j\rangle =\delta_{i,j}$. The \emph{space of shapes of rank-$r$ lattices} is defined to be the double coset space:
\[
\mathcal{S}_r := \op{GL}_r(\mathbb{Z}) \backslash \op{GL}_r(\mathbb{R}) / \op{GO}_r(\mathbb{R}),
\]
where $\op{GO}_r(\mathbb{R})$ denotes the group of invertible linear transformations preserving $\langle \cdot, \cdot \rangle$ up to a positive scalar multiple. Let $\Lambda\subset V$ be a full rank lattice in $V$ and let $\{b_1,\dots,b_r\}$ be a chosen $\Z$-basis of $\Lambda$. We write $b_i=\sum_{j=1}^r m_{i,j} e_j$ and let $M:=(m_{i,j})\in\op{GL}_r(\mathbb{R})$. Then the class of $M$ in the quotient $\op{GL}_r(\mathbb{Z}) \backslash \op{GL}_r(\mathbb{R})$ depends only on $\Lambda$ and not on the chosen basis. The \emph{shape of $\Lambda$} is defined to be the class of $M$ in the double coset space $\mathcal{S}_r$ and is denoted by $\op{sh}(\Lambda)$. The space of shapes carries a natural measure $\mu_{\mathcal{S}_r}$ induced from Haar measure on $\op{GL}_r(\mathbb{R})$.
\par An equivalent description is obtained via Gram matrices. Let $\mathcal{P}_r$ denote the space of positive definite symmetric $r\times r$ real matrices. Since the basis vectors above are written as the rows of $M$, the group $\op{GL}_r(\mathbb{Z})$ acts on $\mathcal{P}_r$ by $A\cdot G := AGA^{T}$ for $G\in \mathcal{P}_r$ and $A\in \op{GL}_r(\mathbb{Z})$. On the other hand,
$\lambda\in \mathbb{R}^{\times}$ acts by $\lambda\cdot G:=\lambda^2G$. Writing
\[
\mathcal{H}_r := \mathcal{P}_r / \mathbb{R}^{\times}
 = \{\, G \in M_r(\mathbb{R}) \mid G = G^T,\; G>0,\; \det G = 1 \,\},
\]
there is a natural identification
\[
\op{GL}_r(\mathbb{R}) / \op{GO}_r(\mathbb{R}) \;\xrightarrow{\;\sim\;}\; \mathcal{H}_r,\]
mapping $M$ to $|\det M|^{-2/r} MM^T$. Thus we have an isomorphism
\[
\mathcal{S}_r \simeq \op{GL}_r(\mathbb{Z}) \backslash \mathcal{H}_r.
\]
Under this identification, before passing to the arithmetic quotient the Gram matrix of the row basis is $MM^T$. If $A$ changes this row basis and $\lambda$ rescales the lattice, then
\[
\op{Gr}(\lambda A\cdot \Lambda)=\lambda^2 A\,\op{Gr}(\Lambda) A^T.
\]
Consequently its class in $\op{GL}_r(\mathbb{Z})\backslash\mathcal H_r$ is independent of the chosen basis and of scaling.

\par The case $r=2$ admits a particularly transparent description. Writing a general element of $\mathcal{H}_2$ as $\begin{pmatrix} a & b \\ b & c \end{pmatrix}$ where $a>0$ and $ac-b^2=1$, one obtains an explicit $\op{GL}_2(\mathbb{R})$-equivariant identification of $\mathcal{H}_2$ with $\mathbb{H}$, the complex upper half-plane. Explicitly, the matrix above corresponds to the point $z=x+iy\in\mathbb{H}$ with
\[
x=\frac{b}{c}, \qquad y=\sqrt{\frac{a}{c}-x^2}.
\]
Under this correspondence, the action of $\op{GL}_2(\mathbb{Z})$ on $\mathcal{H}_2$ becomes the usual fractional linear action on $\mathbb{H}$, and a fundamental domain is given by
\[
\mathcal{F}_2=\{\, x+iy\in\mathbb{H} \mid x\in[0,1/2),\; y>0,\; x^2+y^2\ge 1 \,\}.
\]
The induced measure on $\mathcal{F}_2$ is the hyperbolic measure $\frac{dx\, dy}{y^2}$ coming from Haar measure on $\op{GL}_2(\mathbb{R})$.
\par Let $K$ be a number field of degree $n$ and signature $(r_1,r_2)$, and let $\mathcal O\subset K$ be an arbitrary order. Choose the real embeddings $\sigma_1,\ldots,\sigma_{r_1}$ and one complex embedding $\tau_j$ from each conjugate pair. The Minkowski embedding
\[
\begin{aligned}
j:K&\longrightarrow \mathbb R^{r_1}\times\mathbb C^{r_2},\\
\alpha&\longmapsto (\sigma_1(\alpha),\ldots,\sigma_{r_1}(\alpha),\tau_1(\alpha),\ldots,\tau_{r_2}(\alpha))
\end{aligned}
\]
identifies $K$ with a lattice subspace of
\[
K_{\mathbb R}:=K\otimes_{\mathbb Q}\mathbb R\simeq \mathbb R^{r_1}\times\mathbb C^{r_2}.
\]
On this real vector space we use the positive-definite Minkowski quadratic form
\[
q(x)=\sum_{j=1}^{r_1}\sigma_j(x)^2+2\sum_{j=1}^{r_2}|\tau_j(x)|^2,
\]
with associated real inner product
\[
\langle x,y\rangle_{\mathrm{M}}=
\sum_{j=1}^{r_1}\sigma_j(x)\sigma_j(y)
+2\operatorname{Re}\sum_{j=1}^{r_2}\tau_j(x)\overline{\tau_j(y)}.
\]
This form is positive definite in every signature and satisfies
$\langle x,1\rangle_{\mathrm{M}}=\op{Tr}_{K/\mathbb Q}(x)$ and
$\langle 1,1\rangle_{\mathrm{M}}=n$. Hence the orthogonal projection to the hyperplane perpendicular to $1$ is
\[
\pi(x)=x-\frac{\op{Tr}_{K/\mathbb Q}(x)}{n}.
\]
If
\[
\mathcal O^\perp:=\{\,n\alpha-\op{Tr}_{K/\mathbb Q}(\alpha):\alpha\in\mathcal O\,\},
\]
then $\mathcal O^\perp=n\pi(\mathcal O)$. Thus $\mathcal O^\perp$ is a scalar multiple of the projected lattice and has the same shape. We define the \emph{shape of the order $\mathcal O$}, denoted $\op{sh}(\mathcal O)$, to be this shape; for $\mathcal O=\mathcal O_K$ it is the shape of $K$. If $\{1,\alpha_1,\ldots,\alpha_{n-1}\}$ is a $\mathbb Z$-basis of $\mathcal O$ and $\alpha_i^\perp=n\alpha_i-\op{Tr}_{K/\mathbb Q}(\alpha_i)$, its unreduced Gram matrix is
\[\label{gram matrix computation}
\operatorname{Gr}\left(\mathcal O^\perp\right)
=\bigl(\langle \alpha_i^\perp,\alpha_j^\perp\rangle_{\mathrm M}\bigr)_{1\leq i,j\leq n-1}.
\]

For quadratic fields, this invariant is trivial, as all rank-$1$ lattices are homothetic. In contrast, for cubic fields the associated projected lattice has rank $2$, and hence the shape of a cubic field is represented by a point in the fundamental domain $\mathcal{F}_2 \subset \mathbb{H}$.
\subsection{Parameterization of order $2$ ideals in cubic orders}

Let $T$ be a principal ideal domain. A ring of rank $n$ over $T$ is defined to be a commutative $T$-algebra for which the underlying $T$-module $(R,+)$ is isomorphic to the free module $T^n$. When $n=2,3,4$ or $5$ we say that $R$ is a quadratic, cubic, quartic or quintic ring over $T$ respectively. A ring $R$ of rank $n$ is said to be \emph{nondegenerate} if $\op{Disc}(R)\neq 0$. 

\par We first recall the classical parametrization of cubic rings due to Delone and Faddeev \cite{DeloneFaddev}, as later extended and reformulated by Gan, Gross, and Savin \cite[Proposition 4.2]{GGS}. Let $f(x,y)=ax^{3}+bx^{2}y+cxy^{2}+dy^{3}$ be a binary cubic form with coefficients in $T$ and let $\gamma\in \op{GL}_2(T)$. Then \(\op{GL}_{2}(T)\) acts on the space of binary cubic forms via the twisted action:
\[
(\gamma\cdot f)(x,y)=\det(\gamma)^{-1} f\big((x,y)\gamma\big).
\]

\begin{theorem}[Delone--Faddeev, Gan--Gross--Savin]
There is a canonical bijection between isomorphism classes of cubic rings $R$ over $T$ and
$\op{GL}_{2}(T)$-equivalence classes of binary cubic forms $f$ with coefficients in $T$.
Under this correspondence, let $R(f)$ be the cubic ring associated to a form $f$. The discriminant of $R(f)$ coincides with the
discriminant of $f$. In particular, \(R(f)\) is nondegenerate if and only if \(\op{Disc}(f)\neq 0\).
\end{theorem}

\begin{proof} We provide a sketch of the proof for the benefit of exposition. To describe the correspondence explicitly, let \(R\) be a cubic ring. Choosing a \(T\)-basis \(\langle 1,\omega,\theta\rangle\) for \(R\), and translating \(\omega\) and \(\theta\) by suitable integers, we may assume that \(\omega\theta\in T\). Such a basis is called \emph{normal}. With respect to a normal basis, multiplication in \(R\) is given by relations of the form
\[
\omega\theta = n, \qquad
\omega^{2} = m + b\omega - a\theta, \qquad
\theta^{2} = \ell + d\omega - c\theta,
\]
for integers \(a,b,c,d,\ell,m,n\). Associativity of multiplication imposes constraints on \(\ell,m,n\), and a direct calculation shows that these constants are uniquely determined by
\[
n=-ad, \qquad m=-ac, \qquad \ell=-bd.
\]
The coefficients \(a,b,c,d\) therefore determine the cubic ring \(R\) uniquely, and one associates to \(R\) the binary cubic form
\[
f(x,y)=ax^{3}+bx^{2}y+cxy^{2}+dy^{3}.
\]

Conversely, starting from a binary cubic form \(f(x,y)=ax^{3}+bx^{2}y+cxy^{2}+dy^{3}\), the above relations define a cubic ring \(R(f)\) with underlying \(T\)-module \(T\oplus T\omega\oplus T\theta\). The associativity relations ensure that this construction is well-defined and inverse to the assignment \(R\mapsto f\). A direct computation shows that the discriminant of the cubic ring \(R(f)\) coincides with the discriminant of the associated binary cubic form,
\[
\op{Disc}(R(f))=\op{Disc}(f)
= b^{2}c^{2}-4ac^{3}-4b^{3}d-27a^{2}d^{2}+18abcd.
\]
\end{proof}
\par
A key input in our arguments is Bhargava’s parametrization
\cite[Theorem~4]{bhargavahighercomplaws2} of ideal classes of order two in cubic
rings via orbits of pairs of ternary quadratic forms. We briefly recall the aspects of this
parametrization that are relevant for our work, following the notation of
\cite{BhargavaVarma}. Let $V_T := T^{2}\otimes \op{Sym}^{2} T^{3}$ be
the space of pairs of ternary quadratic forms over \(T\). Equivalently, \(V_T\)
may be identified with the space of pairs of symmetric \(3\times 3\) matrices
with entries in \(T\),
\[
(A,B)=
\left(
\begin{bmatrix}
a_{11} & a_{12} & a_{13}\\
a_{12} & a_{22} & a_{23}\\
a_{13} & a_{23} & a_{33}
\end{bmatrix},
\begin{bmatrix}
b_{11} & b_{12} & b_{13}\\
b_{12} & b_{22} & b_{23}\\
b_{13} & b_{23} & b_{33}
\end{bmatrix}
\right).
\]
\noindent Let
\[
\widetilde G_T:=\op{GL}_2(T)\times\op{GL}_3(T)
\]
be the full product group acting on $V_T$. Explicitly, for
\(\gamma=\mtx{r}{s}{t}{u}\in \op{GL}_2(T)\), define
\[
\gamma^{*}:=\mtx{r}{-t}{-s}{u}.
\]
Then the action of \(\op{GL}_2(T)\) on \(V_T\) is given by
\[
\gamma\cdot (A,B) := (A,B)\gamma^{*} = (rA - sB,\,-tA + uB),
\]
while for $g\in\op{GL}_3(T)$ one has
\[
g\cdot(A,B):=(gAg^t,gBg^t).
\]
These actions commute. Let
\[
\Delta(T):=\left\{\left(\lambda^{-2}I_2,\lambda I_3\right):\lambda\in T^\times\right\}\subset\widetilde G_T.
\]
Equivalently, its elements are
\[
T_\lambda=
\left(
\begin{bmatrix}
\lambda^{-2}&0\\
0&\lambda^{-2}
\end{bmatrix},
\begin{bmatrix}
\lambda&0&0\\
0&\lambda&0\\
0&0&\lambda
\end{bmatrix}
\right),\qquad \lambda\in T^\times.
\]
Every element of $\Delta(T)$ acts trivially. We therefore write
\[
G_T:=\widetilde G_T/\Delta(T)
\]
for the effective group. Unless explicitly stated otherwise, every orbit and stabilizer below is taken with respect to this effective quotient. In particular, the stabilizers occurring in the $p$-adic masses are finite; the infinite subgroup $\Delta(\mathbb Z_p)$ has already been divided out. To any pair \((A,B)\in V_T\), one associates a binary cubic form $f_{(A,B)}(x,y) := \op{Det}(Ax - By)$.
The coefficients of \(f_{(A,B)}\) are polynomial invariants under the
\(\op{SL}_3\)-action, and the discriminant
\[
\op{Disc}(A,B) := \op{Disc}\bigl(f_{(A,B)}\bigr)
\]
is a homogeneous polynomial of degree \(12\) in the coefficients of \(A\) and
\(B\). This discriminant generates the ring of relative invariants for the
action of \(\op{SL}_2\times \op{SL}_3\). An element \((A,B)\in V_T\) is called
\emph{nondegenerate} if \(\op{Disc}(A,B)\neq 0\).

\par
Via the Delone--Faddeev--Gan--Gross--Savin correspondence, the binary cubic form
associated to a pair of ternary quadratic forms determines a cubic ring over
\(T\). This relationship admits a substantial refinement that incorporates ideal
classes of order two. To state it precisely, fix a set \(S\) of representatives
for the left action of \(T^{\times}\) on \(T\setminus\{0\}\). For an ideal \(I\) of a
cubic ring over \(T\), we denote by \(N_{S}(I)\) the distinguished generator in
\(S\) of the ideal norm of \(I\) viewed as a \(T\)-ideal. 

\par The arithmetic data in the correspondence are triples $(R,I,\delta)$, considered up to a ring isomorphism and a simultaneous rescaling of $I$ and $\delta$. Bases are used in the construction to write down a pair of matrices, but are not part of the triples that we count. More precisely, one chooses a normal basis of $R$, determined by a basis of $R/T$, and a basis of $I$ with compatible orientation. Rescaling $I$ by $\kappa$ requires a corresponding change of its basis, with the orientation convention restored if necessary, and the replacement $\delta\mapsto\kappa^2\delta$. We shall use the following orbit correspondence, which incorporates these choices without retaining them as additional arithmetic data.

\begin{theorem}[Bhargava]
The set of nondegenerate orbits for the action of
$G_T$ on $V_T$ is in canonical bijection with equivalence classes of triples
$(R,I,\delta)$, where $R$ is a nondegenerate cubic ring over $T$, $I$ is an
$R$-ideal that is free of rank three as a $T$-module, and $\delta$ is an
invertible element of $R \otimes \Frac(T)$ satisfying the conditions that
$I^2 \subset (\delta)$ and $N_S(I)^2 = N(\delta)$. Two such triples
$(R,I,\delta)$ and $(R',I',\delta')$ are considered equivalent if there exists
a ring isomorphism $\phi \colon R \to R'$ together with an element
$\kappa \in R' \otimes \Frac(T)$ such that $I' = \kappa\,\phi(I)$ and
$\delta' = \kappa^2 \phi(\delta)$.
\end{theorem}

\begin{proof}
The above result is proven in \cite{bhargavahighercomplaws2} (see also \cite[Theorem~11]{BhargavaVarma}).
\end{proof}
\par Given a nondegenerate pair $(A,B)\in V_{\mathbb{R}}$, the associated conics intersect in $4$ points in $\mathbb{P}^2(\mathbb{C})$. The action of $G_{\mathbb{R}}$ on $V_{\mathbb{R}}$ consists of $3$ nondegenerate orbits $V_{\mathbb{R}}^{(i)}$, where $i=0,1,2$. Here, $V_{\mathbb{R}}^{(i)}$ is the $G_{\mathbb{R}}$ orbit consisting of nondegenerate pairs $(A,B)\in V_{\mathbb{R}}$ for which the points of intersection consist of $i$ pairs of complex zeros and $4-2i$ real zeros. A pair $(A,B)\in V_{\Z}$ is said to be absolutely irreducible if the following conditions are satisfied:
\begin{itemize}
    \item $A$ and $B$ do not possess a common $\mathbb{Q}$-rational zero as conics in $\mathbb{P}^{2}$; and
    \item the binary cubic form $f(x, y)=\operatorname{Det}(A x-B y)$ is irreducible over $\mathbb{Q}$.
\end{itemize}
We set $V_\Z^{(i)}:=V_{\mathbb{R}}^{(i)}\cap V_{\Z}$. Consider an element $(A,B)\in V_\Z^{(i)}$ for which the associated binary cubic form is irreducible and let $(R, I, \delta)$ be the associated triple. If $i=0$ or $2$, then $R$ is an order in a totally real cubic field. On the other hand, if $i=1$ then $R$ is an order in a complex cubic field. If $i=0$ (resp. $i=2$) then $\delta$ is (resp. is not) a totally positive element of $R\otimes \Q$. We refer to \cite[Lemma 21]{BhargavaVarma} and its proof for further details.
\subsection{Class groups and narrow class groups}
\par Consider pairs $(R, I, \delta)$ where $R$ is an order in an $S_3$-cubic number field. We note that such cubic rings make up $100\%$ of all orders in cubic fields. A triple $(R,I,\delta)$ is called \emph{projective} if $I$ is a projective
$R$-module, equivalently an invertible fractional ideal; in this case
$I^{2}=(\delta)$. For a fixed $S_{3}$-cubic order $R$, equivalence classes of
projective triples admit a natural composition law
\[
(R,I,\delta)\circ (R,I',\delta')=(R,II',\delta\delta'),
\]
which makes them into a group, denoted $H(R)$.

Let $U$ be the unit group of $R$, $U^{\mathrm{pos\,norm}}$ the subgroup of units
of positive norm, and $U^{2}$ the subgroup of square units. There is an exact
sequence
\[
1 \longrightarrow \frac{U^{\mathrm{pos\,norm}}}{U^{2}}
\longrightarrow H(R)
\longrightarrow \mathrm{Cl}_{2}(R)
\longrightarrow 1.
\]
If $R$ is an order in a totally real (resp. complex) $S_{3}$-cubic field then $|U^{\mathrm{pos\,norm}}/U^{2}|=4$ (resp. $=2$). Thus we find that if $R$ is an order in a totally real (resp. complex) $S_{3}$-cubic field then
$|H(R)|=4\,|\mathrm{Cl}_{2}(R)|$ (resp. $|H(R)|=2\,|\mathrm{Cl}_{2}(R)|$).

When $R$ is totally real, let $H^{+}(R)\subset H(R)$ be the subgroup of triples
$(R,I,\delta)$ with $\delta$ totally positive. Then, one has the relation
\[
|H^{+}(R)|=|\mathrm{Cl}_{2}^{+}(R)|
\]
cf. \cite[Lemma 15]{BhargavaVarma}.
\par We will also need to count these triples when $R$ is an order in a cyclic cubic field. To do this, first keep the cubic order $R$ itself fixed, and identify two pairs $(I,\delta)$ and $(I',\delta')$ only when
\[
I'=\kappa I,\qquad \delta'=\kappa^2\delta
\quad\text{for some }\kappa\in(R\otimes\mathbb Q)^\times.
\]
Write $H_R$ for the group of projective pairs under this equivalence relation, and, when $R$ is totally real, write $H_R^+$ for the subgroup with $\delta$ totally positive. These are the triples with an identified cubic order, sometimes called marked triples. The unit exact sequence and its totally positive counterpart are statements about this identified order; their proofs therefore apply to every cubic order in a number field. In particular,
\[
|H_R|=\begin{cases}
4|\op{Cl}_2(R)|,&R\text{ totally real},\\
2|\op{Cl}_2(R)|,&R\text{ complex},
\end{cases}
\qquad |H_R^+|=|\op{Cl}_2^+(R)|
\quad(R\text{ totally real}).
\]
For an $S_3$-order, the automorphism group of $R$ is trivial, so this agrees with $H(R)$ and $H^+(R)$ above. For a general $R$, the unmarked triples in the orbit correspondence are the orbits of $\op{Aut}(R)$ on these sets.

\par The reducible marked triples are precisely those for which $\delta$ is a square in $(R\otimes\mathbb Q)^\times$; see \cite[Lemma~16 and Theorem~17]{BhargavaVarma}. Indeed, after writing $\delta=\kappa^2$ and rescaling, such a pair takes the form $(I,1)$ with $I^2=R$. Two such pairs are equivalent only when the ideals are equal, since an element of a number field whose square is $1$ is $1$ or $-1$. Thus these pairs are naturally identified with $\mathcal I_2(R)$, and they all belong to $H_R^+$ when $R$ is totally real. Let $T_R$ be the complement of this reducible subset in $H_R$, and let $T_R^+$ be its complement in $H_R^+$. We have
\[
|T_R|=\begin{cases}
4|\op{Cl}_2(R)|-|\mathcal I_2(R)|,&R\text{ totally real},\\
2|\op{Cl}_2(R)|-|\mathcal I_2(R)|,&R\text{ complex},
\end{cases}
\]
and $|T_R^+|=|\op{Cl}_2^+(R)|-|\mathcal I_2(R)|$ in the totally real case. These are nonnegative cardinalities. The global counting functions below count unmarked orbits, with each orbit counted once; the stabilizer weights in the local masses are a separate matter. Lemma~\ref{lemma:cyclic-error} compares the marked and unmarked counts for cyclic orders and proves that their contributions are negligible.

\section{Volume computations}
\par Let $W\subseteq \mathcal{S}_2$ be a measurable subset whose boundary has measure zero. Given a subset $S$ of $V_{\Z}^{(i)}$ and a positive real number $X$, we denote by $N(S;X)$ the number of $\op{G}_{\Z}$-orbits of irreducible elements $x\in S$ such that $|\op{Disc}(x)|<X$. Given an irreducible $x\in S$ let $\op{sh}(x)$ be the shape of the associated cubic ring. Denote by $N_W(S;X)$ the number of $\op{G}_{\Z}$-orbits of irreducible elements $x\in S$ such that $|\op{Disc}(x)|<X$ and such that $\op{sh}(x)\in W$. 
\par Denote by $\op{GL}_n^{\pm1}$ the subgroup of $\op{GL}_n$ with determinant $+1$ or $-1$ and set $\mathbb{G}_m:=\op{GL}_1$. The space $V_{\mathbb{R}}$ is also considered by Bhargava and Harron \cite{bhargavaharron}. The key difference is that here the relevant shape is the shape of the cubic ring attached to the binary cubic invariant $\op{Det}(Ax-By)$, so it is controlled by the $\op{GL}_2$ factor; in \emph{loc. cit.} the quartic-ring shape is controlled by the $\op{GL}_3$ factor.

For the real volume calculation it is convenient to use
\[
G_{\mathbb{R}}^{\prime}:=\mathbb{G}_{m}(\mathbb{R}) \times \mathrm{GL}_{3}^{ \pm 1}(\mathbb{R}) \times \mathrm{GL}_{2}^{ \pm 1}(\mathbb{R}).
\]
Starting from the full product $\widetilde G_{\mathbb R}$, write
$g_j=|\det g_j|^{1/j}g_j'$ and send $(g_2,g_3)$ to
\[
\left(|\det g_3|^{2/3}|\det g_2|^{1/2},g_3',g_2'\right).
\]
This presentation has the same orbits on $V_{\mathbb R}$ as the effective group $G_{\mathbb R}$, and the kernel of the $G_{\mathbb R}'$-action has four elements. Thus the finite factor $4$ in the real stabilizers below is the kernel multiplicity of this presentation, whereas the $p$-adic stabilizers in Section~4 are stabilizers in the effective quotient $G_{\mathbb Z_p}$.

Let $Q(v)$ denote the unreduced Gram form of the associated trace-zero cubic lattice, computed with the Minkowski inner product of Section~2. We choose $v^{(i)}\in V_{\mathbb R}^{(i)}$ so that
\[
|\op{Disc}(v^{(i)})|=1
\quad\text{and}\quad Q(v^{(i)})=3\sqrt3\,I_2.
\]
Here the scalar in the second condition is forced by our choice of trace-zero basis. To see this, let $G$ be the Minkowski Gram matrix of a cubic basis $(1,\alpha_1,\alpha_2)$, and let $P$ be the Gram matrix of the projections $\pi(\alpha_j)=\alpha_j-\op{Tr}(\alpha_j)/3$. Since $\langle1,1\rangle_{\mathrm M}=3$ and $\det G=|\op{Disc}(v)|$, the Schur complement gives $\det P=|\op{Disc}(v)|/3$. Our trace-zero vectors are $3\pi(\alpha_j)$, so $Q(v)=9P$ and
\[
\det Q(v)=27|\op{Disc}(v)|.
\]
Thus the chosen Gram matrix has the required determinant. Equivalently, its determinant-normalized representative is $I_2$. Such a base point exists: a determinant-one real change of the two-dimensional basis makes any positive-definite Gram matrix a scalar multiple of $I_2$, and the scalar coordinate then adjusts the absolute discriminant to one. If $\gamma=(\lambda,g_3,g_2)\in G_{\mathbb R}'$, then $g_3$ leaves the binary cubic invariant unchanged, the scalar coordinate changes only the homothety class, and $g_2$ changes the row basis of the trace-zero cubic lattice. With the row convention of Section~2 there is therefore a positive scalar $c(\gamma)$ such that
\[
Q(\gamma\cdot v^{(i)})=c(\gamma)\,g_2Q(v^{(i)})g_2^T.
\]
After passing from the unreduced Gram form to the arithmetic shape quotient this gives
\[
\op{sh}(\gamma\cdot v^{(i)})=[g_2]\in \op{GL}_2(\mathbb Z)\backslash\op{GL}_2(\mathbb R)/\op{GO}_2(\mathbb R).
\]
We take
\[
G_{\mathbb Z}'=\{\pm1\}\times\op{GL}_3(\mathbb Z)\times\op{GL}_2(\mathbb Z),
\]
whose image in the effective real group is the integral group used for the orbit counts. For $j=2,3$, choose a measurable fundamental domain $\mathcal D_j$ for $\op{GL}_j(\mathbb Z)\backslash\op{GL}_j^{\pm1}(\mathbb R)$, with boundary of Haar measure zero. We use the product fundamental domain
\[
\mathcal F=\mathbb R_{>0}\times\mathcal D_3\times\mathcal D_2
\]
for $G_{\mathbb Z}'\backslash G_{\mathbb R}'$. The notation $\mathcal Fv^{(i)}$ is understood as a multiset: a point occurs once for every preimage in $\mathcal F$ under $g\mapsto gv^{(i)}$. All volumes of this multiset count these multiplicities. With this convention, set

\[\begin{split}
    &\mathcal{R}_{X}:=\left\{x \in \mathcal{F} v^{(i)}\mid |\operatorname{Disc}(x)|<X\right\},\\
    &\mathcal{R}_{X, W}:=\left\{x \in \mathcal{F} v^{(i)}\mid |\operatorname{Disc}(x)|<X\text{ and }\op{sh}(x)\in W\right\}.
\end{split}\]
Let $4 n_{i}$ be the cardinality of the stabilizer in $G_{\mathbb{R}}^{\prime}$ of $v^{(i)} \in V_{\mathbb{R}}^{(i)}$. Then
\[
n_0 = 24, \qquad n_1 = 4, \qquad n_2 = 8.
\]
The four-element kernel is contained in $G_{\mathbb Z}'$. If the stabilizer of a point in the effective integral group has order $s$, its arithmetic orbit is represented in this multiset with total multiplicity $(4n_i)/(4s)=n_i/s$. In particular, an orbit with trivial effective integral stabilizer occurs $n_i$ times. This is the multiset convention underlying the factor $1/n_i$ in the counting theorem. It does not change the definition of the global counting functions, which count each integral orbit once.
\begin{theorem}
    The following assertions hold:
    \begin{enumerate}
        \item $N\left(V_{\mathbb{Z}}^{(i)} ; X\right)=\frac{1}{n_{i}} \operatorname{Vol}\left(\mathcal{R}_{X}\right)+o(X)=\frac{1}{n_{i}} \operatorname{Vol}\left(\mathcal{R}_{1}\right) \cdot X+o(X)$,
        \item $N_W\left(V_{\mathbb{Z}}^{(i)} ; X\right)=\left(1 / n_{i}\right) \operatorname{Vol}\left(\mathcal{R}_{X, W}\right)+o(X)=\left(1 / n_{i}\right) \operatorname{Vol}\left(\mathcal{R}_{1, W}\right) \cdot X+o(X)$.
    \end{enumerate}
\end{theorem}
\begin{proof}
Before applying the counting theorem, we first clarify the normalization being used. Throughout the paper, \(V_{\mathbb Z}\) denotes the lattice of pairs of symmetric integer matrices used in \cite{BhargavaVarma}. This differs slightly from the lattice of integer-coefficient ternary quadratic forms used in the quartic-ring setting. As explained in \cite[Section~5, footnote~6]{BhargavaVarma}, the former has index \(64\) in the latter, and the two discriminant conventions differ by a factor of \(256\). We normalize Lebesgue measure so that our lattice \(V_{\mathbb Z}\) has covolume one. Thus passing from one convention to the other only changes the ambient measure and discriminant bound by fixed factors. Since the same factors occur in both the numerator and denominator of the shape ratio, they cancel. In particular, the volume quotient below is independent of this choice of convention. The absolute counting constants used later are always those of \cite{BhargavaVarma}, with the present integer-matrix normalization.

\par The first assertion follows from \cite[Theorem~4]{bhargavaharron}, after making this harmless change of normalization. For the second assertion, first assume that \(W\) is bounded inside a fixed fundamental domain for \(\mathcal S_2\) and that its boundary has measure zero. The argument in the proof of \cite[Theorem~5]{bhargavaharron} then applies with the additional requirement that the \(\op{GL}_2\)-shape coordinate lie in \(W\). Since \(\partial W\) has measure zero, this does not introduce any new main-term contribution from the boundary, and the usual cusp contribution remains \(o(X)\). Hence the required asymptotic holds for bounded \(W\).

\par For a general \(W\), we use a direct squeeze argument. Let \(\mathcal F_2\) be a fixed finite-volume fundamental domain and choose bounded truncations \(\mathcal F_2(T)\) which exhaust \(\mathcal F_2\). We take the artificial boundary of each truncation to be half-open, so that, with
\[
\mathcal C_T:=\mathcal F_2\setminus\mathcal F_2(T),
\]
the sets \(\mathcal F_2(T)\) and \(\mathcal C_T\) form a disjoint partition of \(\mathcal F_2\). Put
\[
W_T:=W\cap\mathcal F_2(T),
\]
and choose \(T\) so that \(\mu(\partial W_T)=0\). Since \(W\setminus W_T\subseteq\mathcal C_T\),
\[
0\leq N_W(V_{\mathbb Z}^{(i)};X)-N_{W_T}(V_{\mathbb Z}^{(i)};X)
\leq N_{\mathcal C_T}(V_{\mathbb Z}^{(i)};X).
\]
The last term can be estimated without assuming any shape equidistribution for the cusp. Indeed, the partition of \(\mathcal F_2\) gives
\[
N_{\mathcal C_T}(V_{\mathbb Z}^{(i)};X)
=
N(V_{\mathbb Z}^{(i)};X)
-
N_{\mathcal F_2(T)}(V_{\mathbb Z}^{(i)};X).
\]
The first assertion of the theorem and the bounded case of the second assertion therefore imply
\[
\lim_{X\to\infty}\frac{N_{\mathcal C_T}(V_{\mathbb Z}^{(i)};X)}{X}
=
\frac{1}{n_i}
\left(
\operatorname{Vol}(\mathcal R_1)
-
\operatorname{Vol}(\mathcal R_{1,\mathcal F_2(T)})
\right).
\]
As \(T\to\infty\), the regions \(\mathcal R_{1,\mathcal F_2(T)}\) increase to \(\mathcal R_1\) up to a set of measure zero. Hence, by monotone convergence,
\[
\operatorname{Vol}(\mathcal R_{1,\mathcal F_2(T)})
\longrightarrow
\operatorname{Vol}(\mathcal R_1),
\]
and consequently
\[
\lim_{T\to\infty}\limsup_{X\to\infty}
\frac{N_{\mathcal C_T}(V_{\mathbb Z}^{(i)};X)}{X}=0.
\]
We may now first let \(X\to\infty\) for the bounded set \(W_T\), and then let \(T\to\infty\). Since \(\mu(W_T)\to\mu(W)\), this proves the second assertion for arbitrary \(W\) of the stated type.
\end{proof}

\begin{theorem}\label{volume quotient}
For each real orbit $V_{\mathbb R}^{(i)}$, $i\in\{0,1,2\}$, let
$\mathcal R_1^{(i)}$ and $\mathcal R_{1,W}^{(i)}$ denote the regions defined above using the base point $v^{(i)}$. Then
\[
\frac{\operatorname{Vol}\left(\mathcal{R}_{1,W}^{(i)}\right)}
{\operatorname{Vol}\left(\mathcal{R}_{1}^{(i)}\right)}
=
\frac{\mu(W)}{\mu\left(\mathcal{S}_{2}\right)}.
\]
\end{theorem}
\begin{proof}
Fix $i\in\{0,1,2\}$, and abbreviate $\mathcal R_1^{(i)}$ and $\mathcal R_{1,W}^{(i)}$ to $\mathcal R_1$ and $\mathcal R_{1,W}$. We perform the calculation on the product fundamental domain $\mathcal F$ chosen above. This keeps the arithmetic-domain restriction in the calculation throughout.

\par The invariant measure on the open real orbit is a positive constant times $|\op{Disc}(x)|^{-1}dx$. Equivalently, the change of variables $g\mapsto gv^{(i)}$ has Jacobian proportional to $|\op{Disc}(gv^{(i)})|$ when Haar measure is used on the group; see \cite[Proposition~12]{bhargavaharron}. Applying this change of variables on $\mathcal F$, and counting every preimage as prescribed by the multiset convention, gives a constant $a_i>0$ such that
\[
\operatorname{Vol}(\mathcal R_{1,W})
=a_i\int_{\mathcal F}
\mathbf1_{\{|\op{Disc}(gv^{(i)})|<1\}}
\mathbf1_W(\op{sh}(gv^{(i)}))
|\op{Disc}(gv^{(i)})|\,dg.
\]
The same identity without $\mathbf1_W$ gives $\operatorname{Vol}(\mathcal R_1)$. The constant $a_i$ accounts for the chosen Haar and Lebesgue normalizations. There is no further change of domain or multiplicity between these two integrals.

\par Write $g=(\lambda,g_3,g_2)\in\mathcal F$, so that $\lambda>0$, $g_3\in\mathcal D_3$, and $g_2\in\mathcal D_2$. The discriminant has degree $12$ and the two determinant-$\pm1$ factors preserve its absolute value. Therefore
\[
|\op{Disc}(gv^{(i)})|=\lambda^{12}.
\]
Within this chosen domain the discriminant condition is consequently just $0<\lambda<1$. The normalization $Q(v^{(i)})=3\sqrt3\,I_2$, together with the Gram transformation formula above, gives
\[
\op{sh}(gv^{(i)})=[g_2]\in
\mathcal S_2.
\]
In particular, the scalar and $\op{GL}_3$ coordinates do not enter the shape condition.

\par Choose product Haar measure $dg=d^\times\lambda\,dg_3\,dg_2$. The preceding formulas give
\[
\operatorname{Vol}(\mathcal R_{1,W})
=a_i\left(\int_0^1\lambda^{12}\,d^\times\lambda\right)
\left(\int_{\mathcal D_3}dg_3\right)
\left(\int_{\mathcal D_2}\mathbf1_W([g_2])\,dg_2\right).
\]
For $\mathcal R_1$, the last factor is instead $\int_{\mathcal D_2}dg_2$. These integrals are finite: the scalar integral converges, and the determinant-$\pm1$ arithmetic quotients have finite Haar volume.

\par Finally, the quotient map
\[
\op{GL}_2(\mathbb Z)\backslash\op{GL}_2^{\pm1}(\mathbb R)
\longrightarrow
\op{GL}_2(\mathbb Z)\backslash\op{GL}_2^{\pm1}(\mathbb R)/\op{O}_2(\mathbb R)
\simeq\mathcal S_2
\]
pushes Haar measure forward to a constant multiple of $\mu$. This is the measure on shapes used in Section~2; the compact orthogonal factor contributes the same constant to the restricted and unrestricted integrals. Thus
\[
\frac{\int_{\mathcal D_2}\mathbf1_W([g_2])\,dg_2}
     {\int_{\mathcal D_2}dg_2}
=\frac{\mu(W)}{\mu(\mathcal S_2)}.
\]
Cancelling the common scalar, $\op{GL}_3$, and measure-normalization factors proves the asserted volume ratio.
\end{proof}

\section{Congruence conditions}

In this section we study local conditions at all primes $p$. We shall recall some results from \cite{BhargavaVarma} regarding the distribution of $2$-torsion ideals of cubic rings ordered according to discriminant. Let $S \subset V_{\mathbb{Z}}$ be a $G_{\mathbb{Z}}$-invariant subset defined by
finitely many congruence conditions, encoding prescribed local arithmetic
constraints such as maximality, splitting, or ramification. For each prime $p$,
its closure $S_p \subset V_{\mathbb{Z}_p}$ is a compact open subset stable under
the action of $G_{\mathbb{Z}_p}$.
Here and below,
\[
G_{\mathbb Z_p}
=
\widetilde G_{\mathbb Z_p}/\Delta(\mathbb Z_p).
\]
We normalize the additive Haar measure on \(V_{\mathbb Z_p}\) so that \(V_{\mathbb Z_p}\) has total measure \(1\).

We use the same lattice convention as \cite{BhargavaVarma}, namely the lattice of pairs of symmetric integer matrices. As explained in \cite[Section~5, footnote~6]{BhargavaVarma}, this lattice differs from the lattice of integer-coefficient ternary quadratic forms by an index of \(64\), and the corresponding discriminants differ by a factor of \(256\). These normalization factors are already built into the counting formulas of \cite{BhargavaVarma}. With this convention, the global counting constant is therefore the factor \(1/(2n_i)\) appearing below. To $S_p$ one associates a local $p$-adic mass
\[
M_p(S)
=
\sum_{x \in G_{\mathbb{Z}_p}\backslash S_p}
\frac{1}{\operatorname{Disc}_p(x)\,|\operatorname{Stab}_{G_{\mathbb{Z}_p}}(x)|},
\]
which depends only on the local specification of $S$ at $p$; the resulting Euler
product converges since $S$ is defined by finitely many congruence conditions.

Let $N^{(i)}(S;X)$ denote the number of $G_{\mathbb{Z}}$-equivalence classes of
absolutely irreducible elements in $S \cap V_{\mathbb{Z}}^{(i)}$ with
$|\operatorname{Disc}(A,B)|<X$. Then \cite[Theorem 19]{BhargavaVarma} asserts that:
\[
\lim_{X\to\infty} \frac{N^{(i)}(S;X)}{X}
=
\frac{1}{2n_i}
\prod_p \left(\frac{p-1}{p}\, M_p(S)\right).
\]
Next consider sets defined by infinitely many congruence conditions. For each prime \(p\), let \(\Sigma_p\) be a set of isomorphism
classes of nondegenerate cubic rings over \(\mathbb{Z}_p\). 

\begin{definition}\label{acceptable Sigma}A collection
\(\Sigma=(\Sigma_p)_p\) is said to be \emph{acceptable} if, for all sufficiently
large primes \(p\), the set \(\Sigma_p\) contains every maximal cubic ring over
\(\mathbb{Z}_p\) that is not totally ramified.
\end{definition}
For a cubic order \(R\) over
\(\mathbb{Z}\), we write \(R\in\Sigma\), or say that \(R\) is a \(\Sigma\)-order,
if \(R\otimes\mathbb{Z}_p\in\Sigma_p\) for all primes \(p\). Fix such an acceptable collection \(\Sigma\). For every prime $p$, let
\[
S_p(\Sigma_p):=\{\,x\in V_{\mathbb Z_p}: R(x)\in\Sigma_p\text{ and the associated ideal }I(x)\text{ is invertible}\,\}.
\]
Thus $S_p(\Sigma_p)$ is the local \emph{projective} subset. Let $S=S(\Sigma)$ be the subset of $x\in V_{\mathbb Z}$ whose localization lies in $S_p(\Sigma_p)$ for every $p$. Absolute irreducibility is imposed by the counting functions $N^{(i)}$ and $N_W^{(i)}$, rather than being built into the definition of $S(\Sigma)$. Equivalently, for the associated triple $(R,I,\delta)$ one has $R\in\Sigma$ and $I$ is an invertible fractional $R$-ideal. Projectivity is a local condition, so this is equivalent to imposing invertibility at every prime. In all local masses below, $M_p(S(\Sigma))$ means the mass of this projective local set $S_p(\Sigma_p)$; noninvertible ideals are excluded. Thus,
\[
\begin{split}\lim_{X\to\infty}
\frac{N^{(i)}(S(\Sigma);X)}{X}
=&
\frac{1}{2n_i}
\cdot
\prod_p
\left(
\frac{p-1}{p}\cdot M_p(S(\Sigma))
\right),\\
=&\frac{2}{n_i}
\cdot
\prod_p
\left(
\frac{p-1}{p}
\sum_{R\in\Sigma_p}
\frac{1}{\op{Disc}_p(R)}
\cdot
\frac{1}{\lvert\op{Aut}(R)\rvert}
\right).
\end{split}
\]
cf. \cite[(12)]{BhargavaVarma}. Denote the resulting Euler product by
\[
c_\Sigma
:=
\prod_p
\left(
\frac{p-1}{p}
\sum_{R\in\Sigma_p}
\frac{1}{\op{Disc}_p(R)}
\cdot
\frac{1}{\lvert\op{Aut}(R)\rvert}
\right).
\]

\begin{lemma}\label{lemma 4.2}
The following assertions hold:
\begin{enumerate}
\item
The number of totally real \(\Sigma\)-orders \(\mathcal{O}\) with
\(\lvert\op{Disc}(\mathcal{O})\rvert<X\) is
\[
\frac{1}{12}\,c_\Sigma\,X + o(X).
\]
\item
The number of complex \(\Sigma\)-orders \(\mathcal{O}\) with
\(\lvert\op{Disc}(\mathcal{O})\rvert<X\) is
\[
\frac{1}{4}\,c_\Sigma\,X + o(X).
\]
\end{enumerate}
\end{lemma}
\begin{proof}
    The above result is \cite[Lemma 23]{BhargavaVarma}.
\end{proof}
\par Now fix a measurable set $W$ in $\mathcal{S}_2$ whose boundary has measure $0$ and let $N^{(i)}_W(S;X) $ be the number of $G_{\Z}$-equivalence classes of absolutely irreducible elements $(A,B)$ in $S\cap V_{\Z}^{(i)}$ for which $|\op{Disc}(A,B)|<X$ and $\op{sh}(A,B)\in W$.

\begin{lemma}\label{finitely many conditions}
   When $S$ is defined by finitely many congruence conditions, \[
\lim_{X\to\infty} \frac{N^{(i)}_W(S;X)}{X}
=
\frac{1}{2n_i}\frac{\mu(W)}{\mu(\mathcal{S}_2)}
\prod_p \left(\frac{p-1}{p}\, M_p(S)\right).
\]
\end{lemma}
\begin{proof}
First assume that \(W\) is bounded inside a fixed fundamental domain for \(\mathcal S_2\), and that \(\partial W\) has measure zero. Since \(S\) is defined by finitely many congruence conditions, there is a finite set of primes \(P\) such that no condition is imposed outside \(P\). For each \(p\in P\), let $\mu_p(S)$ be the measure of $S_p$. Thus the total \(p\)-adic density contributed by the congruence conditions is
\[
\mu_{\mathrm{fin}}(S)
:=
\prod_{p\in P}\mu_p(S).
\]

The geometry-of-numbers count with congruence conditions, as in
\cite[Lemma~9]{bhargavaharron}, gives
\[
N^{(i)}(S;X)
=
\frac{1}{n_i}\,
\mu_{\mathrm{fin}}(S)\,
\operatorname{Vol}(\mathcal R_X)
+
o(X).
\]
Since
\[
\operatorname{Vol}(\mathcal R_X)
=
X\,\operatorname{Vol}(\mathcal R_1),
\]
this may be written as
\[
N^{(i)}(S;X)
=
\frac{1}{n_i}\,
\mu_{\mathrm{fin}}(S)\,
\operatorname{Vol}(\mathcal R_1)\,X
+
o(X).
\]

The same argument may be applied after imposing the additional condition
\[
\op{sh}(x)\in W.
\]
The congruence conditions concern the \(p\)-adic coordinates, whereas the shape condition is imposed only on the real \(\op{GL}_2\)-coordinate. Hence the finite product of local densities is unchanged. We obtain
\[
N_W^{(i)}(S;X)
=
\frac{1}{n_i}\,
\mu_{\mathrm{fin}}(S)\,
\operatorname{Vol}(\mathcal R_{1,W})\,X
+
o(X).
\]
Dividing the last two asymptotic formulas gives
\[
\lim_{X\to\infty}
\frac{N_W^{(i)}(S;X)}
     {N^{(i)}(S;X)}
=
\frac{\operatorname{Vol}(\mathcal R_{1,W})}
     {\operatorname{Vol}(\mathcal R_1)}.
\]
By Theorem~\ref{volume quotient},
\[
\frac{\operatorname{Vol}(\mathcal R_{1,W})}
     {\operatorname{Vol}(\mathcal R_1)}
=
\frac{\mu(W)}{\mu(\mathcal S_2)}.
\]
Consequently,
\[
\lim_{X\to\infty}
\frac{N_W^{(i)}(S;X)}
     {N^{(i)}(S;X)}
=
\frac{\mu(W)}{\mu(\mathcal S_2)}.
\]

We now combine this with the unrestricted counting formula
\[
\lim_{X\to\infty}
\frac{N^{(i)}(S;X)}{X}
=
\frac{1}{2n_i}
\prod_p
\left(
\frac{p-1}{p}M_p(S)
\right).
\]
It follows that
\[
\begin{aligned}
\lim_{X\to\infty}
\frac{N_W^{(i)}(S;X)}{X}
&=
\left(
\lim_{X\to\infty}
\frac{N_W^{(i)}(S;X)}
     {N^{(i)}(S;X)}
\right)
\left(
\lim_{X\to\infty}
\frac{N^{(i)}(S;X)}{X}
\right)\\
&=
\frac{\mu(W)}{\mu(\mathcal S_2)}
\cdot
\frac{1}{2n_i}
\prod_p
\left(
\frac{p-1}{p}M_p(S)
\right).
\end{aligned}
\]
This proves the required formula when \(W\) is bounded.

\par We now remove the boundedness assumption. Let \(\mathcal F_2\) be a fixed finite-volume fundamental domain for \(\mathcal S_2\). For \(T>0\), choose a bounded truncation $\mathcal F_2(T)\subset\mathcal F_2$
and set $\mathcal C_T
:=
\mathcal F_2\setminus\mathcal F_2(T)$. The truncations may be chosen so that $\mu(\mathcal C_T)\longrightarrow0$ as $T\to\infty$. Set
\[
W_T:=W\cap\mathcal F_2(T),
\]
choosing \(T\) outside a set of measure zero so that
\[
\mu(\partial W_T)=0.
\]
Then \(W_T\) is bounded and
\[
W\setminus W_T\subseteq\mathcal C_T.
\]
Therefore
\[
0
\leq
N_W^{(i)}(S;X)-N_{W_T}^{(i)}(S;X)
\leq
N_{\mathcal C_T}^{(i)}(S;X).
\]

For each fixed \(T\), the bounded case already proved gives
\[
\lim_{X\to\infty}
\frac{N_{W_T}^{(i)}(S;X)}{X}
=
\frac{1}{2n_i}
\frac{\mu(W_T)}{\mu(\mathcal S_2)}
\prod_p
\left(
\frac{p-1}{p}M_p(S)
\right).
\]
It remains only to show that the contribution from $\mathcal C_T$ disappears as $T\to\infty$. Every orbit counted in this cusp with the congruence conditions is also counted without them, so
\[
0\leq N_{\mathcal C_T}^{(i)}(S;X)
\leq N_{\mathcal C_T}(V_{\mathbb Z}^{(i)};X).
\]
The final part of the proof of Theorem~3.1 has already shown that the right-hand side has arbitrarily small normalized count as $T\to\infty$. Consequently,
\[
\lim_{T\to\infty}
\limsup_{X\to\infty}
\frac{N_{\mathcal C_T}^{(i)}(S;X)}{X}
=
0.
\]
Hence
\[
\lim_{T\to\infty}
\limsup_{X\to\infty}
\frac{
N_W^{(i)}(S;X)-N_{W_T}^{(i)}(S;X)
}{X}
=
0.
\]

Finally,
\[
\mu(W_T)\longrightarrow\mu(W)
\qquad (T\to\infty),
\]
because
\[
0\leq \mu(W)-\mu(W_T)
\leq \mu(\mathcal C_T)\longrightarrow0.
\]
We may therefore first let \(X\to\infty\) and then let \(T\to\infty\). This yields
\[
\lim_{X\to\infty}
\frac{N_W^{(i)}(S;X)}{X}
=
\frac{1}{2n_i}
\frac{\mu(W)}{\mu(\mathcal S_2)}
\prod_p
\left(
\frac{p-1}{p}M_p(S)
\right),
\]
as required.
\end{proof}

\par Before passing to acceptable families, we explain how to allow an arbitrary specification at each of finitely many primes. This is needed because a collection of local orders may have unbounded discriminant valuation, even when only one prime is under consideration.

\par The argument applies both to the lattice $V_{\mathbb Z_p}$ of pairs of ternary quadratic forms and to the lattice of binary cubic forms. Write $L_p$ for either lattice, $D$ for its discriminant polynomial, and $A_p\subset L_p$ for a union of nondegenerate integral orbits. In the first case we use the projective local set $S_p(\Sigma_p)$; in the second we use the forms whose associated cubic ring lies in $\Sigma_p$.

\par Each nondegenerate integral orbit is open in the $p$-adic topology. Indeed, over $\mathbb Q_p$ the acting group has the same dimension as the representation and has a finite stabilizer at a nondegenerate point. Its orbit map therefore has invertible differential, and the $p$-adic inverse function theorem shows that the orbit of the open integral subgroup is open. It follows that both $A_p$ and its complement are open within the nondegenerate locus. This observation also applies when projectivity is required, since it is invariant under the integral group action.

\par For a positive integer $m$, put
\[
K_{p,m}=\{x\in L_p:v_p(D(x))<m\},\qquad
E_{p,m}=\{x\in L_p:D(x)\in p^m\mathbb Z_p\}.
\]
The set $K_{p,m}$ is compact and open and contains only nondegenerate points. Hence $A_p\cap K_{p,m}$ is both open and closed in $K_{p,m}$, and is itself compact open in $L_p$. Define
\[
A_{p,m}^-=A_p\cap K_{p,m},\qquad
A_{p,m}^+=(A_p\cap K_{p,m})\cup E_{p,m}.
\]
These are invariant compact-open sets with
$A_{p,m}^-\subseteq A_p\subseteq A_{p,m}^+$, and their difference is $E_{p,m}$. A compact-open subset of $L_p$ is a finite union of residue classes modulo some power of $p$, so both approximating sets are genuine congruence conditions. Moreover, with additive Haar measure normalized by $\mu_p(L_p)=1$,
\[
\mu_p(E_{p,m})\longrightarrow0\qquad(m\to\infty).
\]
This follows by continuity of measure from above, since the sets $E_{p,m}$ decrease to the zero set of the nonzero polynomial $D$, which has measure zero.

\par Now let $P$ be a finite set of primes. Impose $A_p$ for $p\in P$, with no restriction at other primes, and denote the resulting global set by $A$. Let $A_m^-$ and $A_m^+$ be obtained by using $A_{p,m}^-$ and $A_{p,m}^+$ at these primes. Both are defined by finitely many congruences, and $A_m^-\subseteq A\subseteq A_m^+$. If $N_W$ denotes the corresponding orbit count in a chosen real signature, then
\[
0\leq N_W(A_m^+;X)-N_W(A_m^-;X)
\leq\sum_{p\in P}N(E_{p,m};X),
\]
where each count on the right has no shape restriction and imposes only the indicated divisibility condition. The unrestricted finite-congruence theorem shows that, for each $p$ and $m$, the quotient $N(E_{p,m};X)/X$ tends to $C_L\mu_p(E_{p,m})$, where $C_L$ is the unrestricted counting constant for this lattice and real signature. Thus
\[
\lim_{m\to\infty}\limsup_{X\to\infty}
\frac{N_W(A_m^+;X)-N_W(A_m^-;X)}{X}=0.
\]
Applying Lemma~\ref{finitely many conditions} to the two approximating sets and then letting $m\to\infty$ proves the same shape-counting formula for $A$. For binary cubic forms one uses \cite[Theorems~7 and~10]{bhargavaharron}, with $n=3$, in exactly the same squeeze. The local densities converge because the upper and lower sets differ by $E_{p,m}$; the mass expressions have the same limits by \cite[Lemma~20]{BhargavaVarma} and \cite[Lemma~32]{BhargavaShankarTsimerman}, respectively. This proves the extension for arbitrary local specifications at finitely many primes. When this argument is used below, the limits are taken first in $X$, then in the congruence depth $m$ with the retained primes fixed, and finally in the prime cutoff $Y$.

\begin{lemma}\label{main lemma 1}
Let $\Sigma$ be acceptable and put
\[
\rho_W:=\frac{\mu(W)}{\mu(\mathcal S_2)}.
\]
Then
\[
\lim_{X\to\infty}\frac{N_W^{(i)}(S(\Sigma);X)}{X}
=\frac{2c_\Sigma}{n_i}\rho_W.
\]
\end{lemma}
\begin{proof}
For \(Y\geq 2\), let
\[
S_{\leq Y}
:=
\left\{
x\in V_{\mathbb Z}\ \middle|\
x_p\in S_p(\Sigma_p)\text{ for every prime }p\leq Y
\right\}.
\]
Thus \(S_{\leq Y}\) imposes the required local conditions only at the primes up to \(Y\). Clearly,
\[
S(\Sigma)\subseteq S_{\leq Y},
\]
and
\[
S(\Sigma)=\bigcap_{Y\geq2}S_{\leq Y}.
\]

For each prime \(p\), let
\[
B_p(\Sigma)
:=
V_{\mathbb Z_p}\setminus S_p(\Sigma_p).
\]
An element lies in \(B_p(\Sigma)\) precisely when it fails the required local condition at \(p\). In our situation this means that either the associated cubic ring does not lie in \(\Sigma_p\), or the ideal occurring in the corresponding triple \((R,I,\delta)\) is not invertible. Since we only count absolutely irreducible elements, the degenerate locus does not contribute to the counting functions below.

The difference between \(S_{\leq Y}\) and \(S(\Sigma)\) therefore consists of those elements which satisfy all the required conditions at primes \(p\leq Y\), but fail at least one condition at a prime \(p>Y\). In particular,
\[
S_{\leq Y}\setminus S(\Sigma)
\subseteq
\bigcup_{p>Y}
\left\{
x\in V_{\mathbb Z}:x_p\in B_p(\Sigma)
\right\}.
\]
The uniform tail estimate used in the proof of
\cite[(12)]{BhargavaVarma} shows that these failures at large primes contribute negligibly. More precisely,
\[
\lim_{Y\to\infty}
\limsup_{X\to\infty}
\frac{
N^{(i)}(S_{\leq Y};X)
-
N^{(i)}(S(\Sigma);X)
}{X}
=
0.
\]
This is the point at which acceptability of the collection
\(\Sigma=(\Sigma_p)_p\) is used.

We now impose the shape condition. Since
\[
S(\Sigma)\subseteq S_{\leq Y},
\]
we have
\[
N_W^{(i)}(S(\Sigma);X)
\leq
N_W^{(i)}(S_{\leq Y};X).
\]
Moreover, every orbit counted by
\[
N_W^{(i)}(S_{\leq Y};X)
-
N_W^{(i)}(S(\Sigma);X)
\]
is also counted by the corresponding difference without the shape restriction. Hence
\[
0
\leq
N_W^{(i)}(S_{\leq Y};X)
-
N_W^{(i)}(S(\Sigma);X)
\leq
N^{(i)}(S_{\leq Y};X)
-
N^{(i)}(S(\Sigma);X).
\]
After dividing by \(X\), the preceding tail estimate gives
\[
\lim_{Y\to\infty}
\limsup_{X\to\infty}
\frac{
N_W^{(i)}(S_{\leq Y};X)
-
N_W^{(i)}(S(\Sigma);X)
}{X}
=
0.
\]
Thus the shape-restricted count for the full acceptable family can be recovered from the truncated families \(S_{\leq Y}\).

For fixed $Y$, the conditions defining $S_{\leq Y}$ are imposed at finitely many primes, but their congruence depths need not be bounded. The preceding approximation argument extends Lemma~\ref{finitely many conditions} to precisely this situation, including the projectivity requirement at the retained primes. We therefore obtain
\[
\lim_{X\to\infty}
\frac{N_W^{(i)}(S_{\leq Y};X)}{X}
=
\frac{\rho_W}{2n_i}
\prod_p
\left(
\frac{p-1}{p}M_p(S_{\leq Y})
\right),
\]
where
\[
\rho_W
=
\frac{\mu(W)}{\mu(\mathcal S_2)}.
\]
This is the unrestricted Euler product with finitely many local factors modified. To make the normalization clear, put $M_p^{\mathrm{all}}=M_p(V_{\mathbb Z_p})$ and
\[
C_i^{\mathrm{all}}=\frac{1}{2n_i}
\prod_p\left(\frac{p-1}{p}M_p^{\mathrm{all}}\right).
\]
The constant in the last asymptotic can then be written as
\[
\rho_W C_i^{\mathrm{all}}
\prod_{p\leq Y}\frac{M_p(S_p(\Sigma_p))}{M_p^{\mathrm{all}}}.
\]
Only this product of ratios is finite. The unrestricted factors at $p>Y$ remain part of $C_i^{\mathrm{all}}$. Here $M_p(S_p(\Sigma_p))$ denotes the mass of the indicated local projective set, using the same definition as above.

We now let \(Y\to\infty\). By the definition of the acceptable family and the convergence of the projective local mass product,
\[
\lim_{Y\to\infty}
\frac{1}{2n_i}
\prod_p
\left(
\frac{p-1}{p}M_p(S_{\leq Y})
\right)
=
\frac{2c_\Sigma}{n_i}.
\]
Equivalently, this is the same limiting constant that occurs in the unrestricted count
\[
\lim_{X\to\infty}
\frac{N^{(i)}(S(\Sigma);X)}{X}
=
\frac{2c_\Sigma}{n_i}.
\]

We can now compare the truncated and full shape-restricted counts. For every fixed \(Y\),
\[
\frac{N_W^{(i)}(S(\Sigma);X)}{X}
=
\frac{N_W^{(i)}(S_{\leq Y};X)}{X}
-
\frac{
N_W^{(i)}(S_{\leq Y};X)
-
N_W^{(i)}(S(\Sigma);X)
}{X}.
\]
First let \(X\to\infty\). The first term tends to
\[
\frac{\rho_W}{2n_i}
\prod_p
\left(
\frac{p-1}{p}M_p(S_{\leq Y})
\right),
\]
while the second term is controlled by the unrestricted tail above. We then let \(Y\to\infty\). The tail term disappears, and the Euler products with finitely many modified factors converge to the full projective mass product. Therefore
\[
\lim_{X\to\infty}
\frac{N_W^{(i)}(S(\Sigma);X)}{X}
=
\rho_W\frac{2c_\Sigma}{n_i}.
\]
In other words,
\[
\lim_{X\to\infty}
\frac{N_W^{(i)}(S(\Sigma);X)}{X}
=
\frac{2c_\Sigma}{n_i}
\frac{\mu(W)}{\mu(\mathcal S_2)}.
\]
\end{proof}

\begin{lemma}\label{main lemma 2}
The following assertions hold:
\begin{enumerate}
\item
The number of totally real cubic \(\Sigma\)-orders \(\mathcal{O}\) with shape in $W$ and
\(\lvert\op{Disc}(\mathcal{O})\rvert<X\) is asymptotic to
\[
\frac{c_\Sigma}{12}\frac{\mu(W)}{\mu(\mathcal{S}_2)}X.
\]
\item
The number of complex cubic \(\Sigma\)-orders \(\mathcal{O}\) with shape in $W$ and
\(\lvert\op{Disc}(\mathcal{O})\rvert<X\) is asymptotic to
\[
\frac{c_\Sigma}{4}\frac{\mu(W)}{\mu(\mathcal{S}_2)} X.
\]
\end{enumerate}
\end{lemma}

\begin{proof}
Put
\[
\rho_W:=\frac{\mu(W)}{\mu(\mathcal S_2)}.
\]
For a sign $\varepsilon\in\{+,-\}$, let $C_{\varepsilon,W}(\Sigma;X)$ denote the number of cubic $\Sigma$-orders of discriminant sign $\varepsilon$, absolute discriminant less than $X$, and shape in $W$. We write $C_{\varepsilon}(\Sigma;X)$ for the same count without the shape condition.

For $Y\geq2$, let $\Sigma_{\leq Y}$ be the family obtained by imposing the prescribed condition $\Sigma_p$ only for primes $p\leq Y$, with no local restriction at primes $p>Y$. 
For fixed $Y$, the binary-cubic version of the approximation argument above extends the finite-congruence shape-counting theorem to this family. Indeed, one first bounds the discriminant valuation at each prime $p\leq Y$, applies that theorem to the resulting compact-open conditions, and then removes these valuation bounds. It follows that
\[
\frac{C_{+,W}(\Sigma_{\leq Y};X)}{X}
\longrightarrow
\rho_W\,\frac{c_{\Sigma_{\leq Y}}}{12},
\qquad
\frac{C_{-,W}(\Sigma_{\leq Y};X)}{X}
\longrightarrow
\rho_W\,\frac{c_{\Sigma_{\leq Y}}}{4}
\]
as $X\to\infty$. 
Here $c_{\Sigma_{\leq Y}}$ is the Euler-product constant for the truncated family: it uses the prescribed cubic-order mass at $p\leq Y$ and the unrestricted cubic-order mass at $p>Y$. Thus it too is an unrestricted Euler product with finitely many modified factors. The corresponding unrestricted formulas are
\[
\frac{C_{+}(\Sigma_{\leq Y};X)}{X}
\longrightarrow
\frac{c_{\Sigma_{\leq Y}}}{12},
\qquad
\frac{C_{-}(\Sigma_{\leq Y};X)}{X}
\longrightarrow
\frac{c_{\Sigma_{\leq Y}}}{4}.
\]

It remains to pass from the truncated family to the full acceptable family. Since $\Sigma\subseteq\Sigma_{\leq Y}$, for either sign we have
\[
0\leq
C_{\varepsilon,W}(\Sigma_{\leq Y};X)-C_{\varepsilon,W}(\Sigma;X)
\leq
C_{\varepsilon}(\Sigma_{\leq Y};X)-C_{\varepsilon}(\Sigma;X).
\]
The required uniform estimate is \cite[Proposition~29]{BhargavaShankarTsimerman}, used in the sieve in \emph{loc. cit.}, Section~8.4; the resulting cubic-order asymptotic is quoted in \cite[Lemma~23]{BhargavaVarma}. For all sufficiently large $p$, acceptability ensures that an order failing the specification at $p$ is either nonmaximal there or maximal and totally ramified. Proposition~29 bounds the number of such cubic orders by $O(X/p^2)$, uniformly in $p$. A union bound therefore bounds the difference of the unrestricted counts by
\[
C_\varepsilon(\Sigma_{\leq Y};X)-C_\varepsilon(\Sigma;X)
\ll X\sum_{p>Y}p^{-2}
\]
for sufficiently large $Y$. In particular,
\[
\lim_{Y\to\infty}\limsup_{X\to\infty}
\frac{C_{\varepsilon}(\Sigma_{\leq Y};X)-C_{\varepsilon}(\Sigma;X)}{X}=0.
\]
Consequently the same double limit is zero with $C_{\varepsilon,W}$ in place of $C_{\varepsilon}$. This is precisely the same truncation mechanism used in Lemma~\ref{main lemma 1}: the shape restriction can only decrease the large-prime tail.

Finally, acceptability also implies convergence of the truncated Euler products,
\[
c_{\Sigma_{\leq Y}}\longrightarrow c_\Sigma.
\]
For fixed $Y$, subtract the error
$C_{\varepsilon,W}(\Sigma_{\leq Y};X)-C_{\varepsilon,W}(\Sigma;X)$ from the truncated asymptotic, first let $X\to\infty$, and then let $Y\to\infty$. We obtain
\[
\frac{C_{+,W}(\Sigma;X)}{X}
\longrightarrow
\rho_W\,\frac{c_\Sigma}{12},
\qquad
\frac{C_{-,W}(\Sigma;X)}{X}
\longrightarrow
\rho_W\,\frac{c_\Sigma}{4}.
\]
These are the two asserted asymptotics.
\end{proof}

\begin{lemma}\label{lemma:cyclic-error}
Let $\Sigma$ be acceptable. With the shape condition $\op{sh}(R)\in W$ imposed, the contribution of cyclic cubic orders to each of the torsion-weighted orbit counts below is $o(X)$. More precisely,
\[
\begin{aligned}
N_W^{(0)}(S(\Sigma);X)+N_W^{(2)}(S(\Sigma);X)
&=
\sum_{\substack{R\in\Sigma,\ R\text{ an }S_3\text{-order}\\
0<\op{Disc}(R)<X,\ \op{sh}(R)\in W}}
\bigl(4|\op{Cl}_2(R)|-|\mathcal I_2(R)|\bigr)+o(X),\\
N_W^{(1)}(S(\Sigma);X)
&=
\sum_{\substack{R\in\Sigma,\ R\text{ an }S_3\text{-order}\\
-X<\op{Disc}(R)<0,\ \op{sh}(R)\in W}}
\bigl(2|\op{Cl}_2(R)|-|\mathcal I_2(R)|\bigr)+o(X),\\
N_W^{(0)}(S(\Sigma);X)
&=
\sum_{\substack{R\in\Sigma,\ R\text{ an }S_3\text{-order}\\
0<\op{Disc}(R)<X,\ \op{sh}(R)\in W}}
\bigl(|\op{Cl}_2^+(R)|-|\mathcal I_2(R)|\bigr)+o(X).
\end{aligned}
\]
The same $o(X)$ estimate holds for the corresponding torsion-weighted sums over cyclic cubic orders themselves.
\end{lemma}
\begin{proof}
For an $S_3$-cubic order, the automorphism group is trivial, so the marked triples of Section~2.3 are already the unmarked orbits counted by $N_W^{(i)}$. Their cardinalities are the differences displayed in the statement. It remains to control both kinds of count for cyclic orders.

\par We first record the precise quartic counting result used for this purpose. In the integer-coefficient ternary-quadratic-form convention of \cite{bhargavequartic}, Theorem~7 counts all absolutely irreducible integral orbits, which correspond to orders in $S_4$- or $A_4$-quartic fields together with their cubic resolvents. Theorem~22 of the same paper counts the pairs consisting of an $S_4$-quartic order and a cubic resolvent. The leading constants agree for each real signature. Subtracting these two asymptotics shows that the number of $A_4$-quartic order--resolvent pairs of absolute discriminant less than $X$ is $o(X)$. This is a count of orders together with their resolvents, as required by the orbit parametrization.

\par The integer-matrix lattice used here is an invariant sublattice of that integer-coefficient lattice, with the same integral orbit action after identifying the effective groups. The discriminant in the quartic convention is $256$ times ours; see \cite[Section~5, footnote~6]{BhargavaVarma}. Thus the orbits in our lattice with cyclic cubic invariant and absolute discriminant less than $X$ inject into the $A_4$ order--resolvent count with cutoff $256X$. Their number is consequently $o(X)$ as well. The extra projectivity, local, and shape restrictions only decrease this number.

\par Now let $R$ be an order in a cyclic cubic field $K$. Every automorphism of $R$ extends to an automorphism of $K$, and hence $|\op{Aut}(R)|\leq3$. With $T_R$ as in Section~2.3, the number of unmarked absolutely irreducible projective triples with cubic order isomorphic to $R$ is $|T_R/\op{Aut}(R)|$. Each orbit of this finite group has at most three elements, so
\[
|T_R/\op{Aut}(R)|\leq |T_R|
\leq3|T_R/\op{Aut}(R)|.
\]
The same inequalities hold for $T_R^+$, since total positivity is preserved by ring automorphisms. For cyclic $R$ the field is totally real, and the cardinalities identified in Section~2.3 give
\[
|T_R|=4|\op{Cl}_2(R)|-|\mathcal I_2(R)|,
\qquad
|T_R^+|=|\op{Cl}_2^+(R)|-|\mathcal I_2(R)|.
\]
Summing the inequalities over cyclic $\Sigma$-orders of discriminant less than $X$ and shape in $W$ bounds each of these sums by three times the corresponding unmarked cyclic-orbit count. The latter is $o(X)$ by the preceding paragraph. There are no cyclic cubic fields of negative discriminant, so no cyclic contribution occurs in the complex case.

\par We have therefore bounded the particular class-group-minus-ideal-group expressions needed in the lemma, as well as the unweighted exceptional orbit counts. Adding these negligible contributions to the exact identities for $S_3$-orders proves all the assertions.
\end{proof}

\begin{theorem}\label{main thm of article}
    For positive discriminant,
\begin{equation}
\label{eq:avg-positive}
\lim_{X\rightarrow \infty} \frac{
\sum_{\substack{R\in\Sigma \\ 0<\op{Disc}(R)<X\\
\op{sh}(R)\in W}}
\bigl(\lvert\op{Cl}_2(R)\rvert - \tfrac14 \lvert\mathcal{I}_2(R)\rvert\bigr)
}{
\sum_{\substack{R\in\Sigma \\ 0<\op{Disc}(R)<X\\
\op{sh}(R)\in W}} 1
}
=1.
\end{equation}
Similarly, for negative discriminant,
\begin{equation}
\label{eq:avg-negative}
\lim_{X\rightarrow \infty} \frac{
\sum_{\substack{R\in\Sigma \\ -X<\op{Disc}(R)<0 \\ \op{sh}(R)\in W}}
\bigl(\lvert\op{Cl}_2(R)\rvert - \tfrac12 \lvert\mathcal{I}_2(R)\rvert\bigr)
}{
\sum_{\substack{R\in\Sigma \\ -X<\op{Disc}(R)<0 \\ \op{sh}(R)\in W}} 1
}
=1.
\end{equation}
\noindent Finally, restricting to totally real rings, we have
\begin{equation}
\label{eq:avg-totally-real}
\lim_{X\rightarrow \infty} \frac{
\sum_{\substack{R\in\Sigma \\ 0<\op{Disc}(R)<X\\\op{sh}(R)\in W}}
\bigl(\lvert\op{Cl}_2^{+}(R)\rvert-\lvert\mathcal I_2(R)\rvert\bigr)
}{
\sum_{\substack{R\in\Sigma \\ 0<\op{Disc}(R)<X \\ \op{sh}(R)\in W}} 1
}
=1.
\end{equation}
\end{theorem}
\begin{proof}
We follow the argument of \cite[(17)]{BhargavaVarma}, but keep the cyclic cubic orders separate. Let $A_+(X)$, $A_-(X)$, and $A_+^{\mathrm{nar}}(X)$ denote the numerators in \eqref{eq:avg-positive}, \eqref{eq:avg-negative}, and \eqref{eq:avg-totally-real}, respectively. By Lemma~\ref{lemma:cyclic-error}, the $S_3$-order identities are exact up to an $o(X)$ contribution from cyclic orders, and the torsion-weighted cyclic sums themselves are $o(X)$. Hence
\[
4A_+(X)
=
N_W^{(0)}(S(\Sigma);X)+N_W^{(2)}(S(\Sigma);X)+o(X),
\]
\[
2A_-(X)
=
N_W^{(1)}(S(\Sigma);X)+o(X),
\]
and
\[
A_+^{\mathrm{nar}}(X)
=
N_W^{(0)}(S(\Sigma);X)+o(X).
\]
Thus no exact class-group identity is being asserted for cyclic cubic orders; their contribution is absorbed into the displayed error terms.
Put $\rho_W=\mu(W)/\mu(\mathcal S_2)$, and let $C_+(X)$ and $C_-(X)$ denote, respectively, the denominators in the positive- and negative-discriminant averages. Lemma~\ref{main lemma 2} gives
\[
\frac{C_+(X)}{X}\longrightarrow \frac{c_\Sigma}{12}\rho_W,
\qquad
\frac{C_-(X)}{X}\longrightarrow \frac{c_\Sigma}{4}\rho_W.
\]
Since $n_0=24$, $n_1=4$, and $n_2=8$, Lemma~\ref{main lemma 1} gives
\[
\frac{N_W^{(0)}}{X}\longrightarrow\frac{c_\Sigma}{12}\rho_W,
\qquad
\frac{N_W^{(1)}}{X}\longrightarrow\frac{c_\Sigma}{2}\rho_W,
\qquad
\frac{N_W^{(2)}}{X}\longrightarrow\frac{c_\Sigma}{4}\rho_W,
\]
where each $N_W^{(i)}$ is evaluated at $(S(\Sigma);X)$. Consequently,
\[
\frac{N_W^{(0)}+N_W^{(2)}}{C_+(X)}\longrightarrow4,
\qquad
\frac{N_W^{(1)}}{C_-(X)}\longrightarrow2\quad \text{and}\quad
\frac{N_W^{(0)}}{C_+(X)}\longrightarrow1.
\]
Since $C_+(X)$ and $C_-(X)$ are both asymptotic to positive constants times $X$, the $o(X)$ terms above disappear after division by the relevant denominator. The three displayed relations therefore yield \eqref{eq:avg-positive}, \eqref{eq:avg-negative}, and \eqref{eq:avg-totally-real}, respectively. As a normalization check, if $W=\mathcal S_2$, then $\rho_W=1$ and these are exactly the unrestricted constants and ratios of \cite[(12), (17)]{BhargavaVarma}. This completes the proof.
\end{proof}

\begin{proof}[Proof of Theorem \ref{thm a}]
Let $P$ be the finite set of primes at which Theorem~\ref{thm a} prescribes local conditions. For $p\in P$, take $\Sigma_p$ to consist of the maximal cubic $\mathbb Z_p$-orders satisfying those conditions; for $p\notin P$, take all maximal cubic $\mathbb Z_p$-orders. This is an acceptable collection, and its global orders are precisely the rings of integers of the cubic fields in the stated family. Theorem 17 and Corollary 18 of \cite{BhargavaVarma} together imply that for any maximal cubic order, $|\mathcal{I}_2(R)|=1$. Hence \eqref{eq:avg-positive} gives $1+\tfrac14=\tfrac54$, \eqref{eq:avg-negative} gives $1+\tfrac12=\tfrac32$, and the corrected \eqref{eq:avg-totally-real} gives $1+1=2$ for the narrow class-group average. The result then immediately follows from Theorem \ref{main thm of article}.
\end{proof}

\begin{proof}[Proof of Theorem \ref{thm c}]
   For the unrestricted family this is exactly \cite[Corollary~3]{BhargavaVarma}. For the family satisfying the fixed finite local specifications in Theorem~\ref{thm c}, one repeats the same elementary argument, using Theorem~\ref{thm a} in that subfamily: the relevant average sizes are unchanged, so the same lower bounds for the proportions of fields with odd class number follow.
\end{proof}

\bibliographystyle{alpha}
\bibliography{references}

\end{document}